\documentclass[12pt]{amsart}

\usepackage[dvipdfmx]{graphicx}
\usepackage{cite}
\usepackage{amsthm, amsmath, amssymb, mathtools}
\usepackage[top=45truemm, bottom=45truemm, left=30truemm, right=30truemm]{geometry}

\renewcommand{\epsilon}{\varepsilon}

\numberwithin{equation}{section}
\newtheorem{theorem}{Theorem}[section]
\newtheorem{lemma}[theorem]{Lemma}
\newtheorem{corollary}[theorem]{Corollary}

\newtheorem{proposition}[theorem]{Proposition}

\newtheorem{example}[theorem]{Example}

\theoremstyle{definition}
\newtheorem{remark}[theorem]{Remark}
\newtheorem{notation}[theorem]{Notation}

\newcommand{\diam}{\mathrm{diam}}
\newcommand{\Haus}{\mathrm{dim}_{\mathrm{H}}\:}
\newcommand{\cH}{\mathcal{H}}
\newcommand{\PS}{\mathrm{PS}}

\title[Linear equations with two variables in Piatetski-Shapiro sequences]{Linear equations with two variables\\ in Piatetski-Shapiro sequences}
\author[Kota Saito]{Kota Saito}
\date{}

\address{Kota Saito\\
Graduate School of Mathematics\\ Nagoya University\\ Furo-cho\\ Chikusa-ku\\ Nagoya\\ 464-8602\\ Japan}
\curraddr{}
\email{m17013b@math.nagoya-u.ac.jp}

\subjclass[2010]{Primary:11D04, Secondary:11K55}
\keywords{Piatetski-Shapiro sequence, Hausdorff dimension, Diophantine equation, Diophantine approximation, discrepancy}
\date{}
\begin{document}
\maketitle
\begin{abstract}
For every non-integral $\alpha>1$, the sequence of the integer parts of $n^{\alpha}$ $(n=1,2,\ldots)$ is called the Piatetski-Shapiro sequence with exponent $\alpha$, and let $\PS(\alpha)$ denote the set of all those terms. For all $X\subseteq \mathbb{N}$, we say that an equation $y=ax+b$ is solvable in $X$ if the equation has infinitely many solutions of distinct pairs $(x,y)\in X^2$. Let $a,b\in \mathbb{R}$ with $a\neq 1$ and $0\leq b<a$, and suppose that the equation $y=ax+b$ is solvable in $\mathbb{N}$. We show that for all $1<\alpha<2$ the equation $y=ax+b$ is solvable in $\PS(\alpha)$. Further, we investigate the set of $\alpha \in (s,t)$ so that the equation $y=ax+b$ is solvable in $\PS(\alpha)$ where $2< s <t$. Finally, we show that the Hausdorff dimension of the set is coincident with $2/s$.
\end{abstract}

\section{Introduction}
For all $x\in \mathbb{R}$, we define $\lfloor x\rfloor$ as the integer part of $x$, and $\{x\}$ as the fractional part of $x$. For every non-integral $\alpha >1 $, the sequence $(\lfloor n^\alpha \rfloor )_{n=1}^\infty$ is called the \textit{Piatetski-Shapiro sequence with exponent} $\alpha$, and let $\PS(\alpha)$ be the set of all those terms. For all $X\subseteq \mathbb{N}$, and for all polynomials $f(x_1,\ldots, x_n)$ with real coefficients, we say that an equation $f(x_1,\ldots, x_n)=0$ is solvable in $X$ if the equation has infinitely many solutions $(x_1,\ldots, x_n )\in X^n$ with $\#\{x_1,\ldots ,x_n\}=n$. In this article, we discuss the solvability in $\PS(\alpha)$ of the equation
\begin{equation}\label{Equation-main1}
y=ax+b
\end{equation}
for fixed $a,b\in \mathbb{R}$ with $a\notin\{0,1\}$. Glasscock asserted that if equation \eqref{Equation-main1} is solvable in $\mathbb{N}$, then for Lebesgue almost every $\alpha>1$, it is solvable or not in $\PS(\alpha)$, according to $\alpha<2$ or $\alpha>2$  \cite{Glasscock17, Glasscock20}. In addition, as a corollary, he showed that for Lebesgue almost every $1<\alpha<2$, there are infinitely many $(k,\ell,m)\in \mathbb{N}^3$ such that 
 \begin{equation}\label{Sequence-klm}
 k,\ \ell,\ m,\ k+\ell,\ \ell+m,\ m+k,\ k+\ell+m
\end{equation}
 are in $\PS(\alpha)$ \cite[Corollary~10]{Glasscock17}. It is a long-standing open problem whether there exists $(k,\ell,m)\in \mathbb{N}^3$ such that all terms of \eqref{Sequence-klm} are in $\PS(2)$ which is the set of all squares.     
 
The goal of this article is to put forward the following theorem which is an improvement on Glasscock's result in the case when $0\leq b<a$. Here for all $F\subseteq \mathbb{R}$, we let $\Haus F$ denote the Hausdorff dimension of $F$. We will give this definition in Section~\ref{Section-Preparations}.

\begin{theorem}\label{Theorem-main1}
Let $a,b\in \mathbb{R}$, with $a \neq 1$ and $0\leq b<a$. Assume that the equation $y=ax+b$ is solvable in $\mathbb{N}$. Then for all $1<\alpha <2$, the equation $y=ax+b$ is solvable in $\PS(\alpha)$. Moreover, for all $s,t\in \mathbb{R}$ with $2< s<t$, we have
\[
\dim_{\mathrm{H}}\{\alpha \in (s,t) \colon \text{$y=ax+b$ is solvable in $\PS(\alpha)$} \}=2/s.
\]
\end{theorem}
We avail of two main improvements when $0\leq b<a$. Firstly, in the case when $1<\alpha<2$, we arrive at the same conclusion as Glasscock's result even if we replace ``for Lebesgue almost every'' with ``for all''.  Secondly, in the case when $\alpha>2$, his result is equivalent to stating that the set $\{\alpha \in (s,t) \colon \text{$y=ax+b$ is solvable in $\PS(\alpha)$} \}$ and has Lebesgue measure $0$ for all $2< s<t$. However, from Theorem~\ref{Theorem-main1}, we find that the set has a Hausdorff dimension of exactly $2/s$. Hence we can discern more details concerning the geometric structure of the set. We will show Theorem~\ref{Theorem-main1} in Section~\ref{Section-Proof_of_Theorem}. 

From the first improvement, we obtain the following:
\begin{corollary}\label{Corollary-main2}
For all $1<\alpha<2$, there are infinitely many $(k,\ell,m)\in \mathbb{N}^3$ such that all of $k$, $\ell$, $m$, $k+\ell$, $\ell+m$, $m+k$, and $k+\ell+m$ are in $\PS(\alpha)$. 
\end{corollary}
  In the proof  of Corollary~1 in \cite{Glasscock17}, Glasscock applies the result given by Frantzikinakis and Wierdl \cite[Proposition~5.1]{FrantzikinakisWierdl}, and shows that if $y=2x$ is solvable in $\PS(\alpha)$ for fixed $1<\alpha<2$, then there are infinitely many $(k,m,\ell)\in \PS(\alpha)^3$ such that all terms of \eqref{Sequence-klm} are in $\PS(\alpha)$. By this proof, we obtain Corollary~\ref{Corollary-main2} by Theorem~\ref{Theorem-main1}.
  
We next discuss the solvability in $\PS(\alpha)$ of the equation
\begin{equation}\label{Equation-main2}
ax+by=cz
\end{equation}
for fixed $a,b,c\in \mathbb{N}$. As a corollary of Theorem~\ref{Theorem-main1}, the following holds:
\begin{corollary}\label{Corollary-ax+by=cz}
Let $a,b,c\in \mathbb{N}$ with $\gcd(a,c)|b$ and $a>b$. Then, for all $1<\alpha<2$, the equation $ax+by=cz$ is solvable $\PS(\alpha)$. Further, for all $2< s<t$, we have
\begin{equation}\label{Inequality-dimension}
\Haus \{\alpha \in (s,t) \colon \text{$ax+by=cz$ is solvable in $\PS(\alpha)$} \} \geq \frac{2}{s}.
\end{equation}
\end{corollary}
Indeed, from the condition $\gcd(a,c)|b$, the equation $ax+b=cz$ is solvable in $\mathbb{N}$. By dividing both sides by $c$, we have the equation $z=(a/c)x+(b/c)$ whose coefficients $a/c$ and $b/c$ satisfy the conditions in Theorem~\ref{Theorem-main1}. Moreover, if the equation $ax+b=cz$ is solvable in $\PS(\alpha)$, then by letting $y=1=\lfloor 1^\alpha \rfloor$, we see that the equation $ax+by=cz$ is solvable in $\PS(\alpha)$. Therefore we conclude Corollary~\ref{Corollary-ax+by=cz} from Theorem~\ref{Theorem-main1}.

 In \cite{MatsusakaSaito}, it is proved that for all $a,b,c\in \mathbb{N}$ and $s,t\in \mathbb{R}$ with $2<s<t$, the left-hand side of \eqref{Inequality-dimension} is greater than or equal to
\begin{align*}
\begin{cases}
\displaystyle  \left(s + \frac{s^3}{(2+\{s\}-2^{1-\lfloor s\rfloor} )(2-\{s\}) } \right)^{-1}  & \text{if $a=b=c$}\\[10pt]
\displaystyle 2\left(s+\frac{s^3}{(2+\{s\}-2^{1-\lfloor s\rfloor} )(2-\{s\}) } \right)^{-1} & \text{otherwise}.
\end{cases}
\end{align*}
The lower bounds \eqref{Inequality-dimension} in Corollary~\ref{Inequality-dimension} are better than the above for all $2<s<t$. In particular, we find that the left-hand side of \eqref{Inequality-dimension} goes to $1$ as $s\rightarrow 2+0$ from Corollary~\ref{Inequality-dimension} if $a,b,c$ are restricted.

\begin{notation}
Let $\mathbb{N}$ be the set of all positive integers, $\mathbb{Z}$ be the set of all integers, $\mathbb{Q}$ be the set of all rational numbers, and $\mathbb{R}$ be the set of all real numbers. For all $x\in \mathbb{R}$, let $\lceil x \rceil$ denote the minimum integer $n$ such that $x\leq n$. For all $a,b\in\mathbb{Z}$, we say $a|b$ if $a$ is a divisor of $b$, and let $\gcd(a,b)$ denote the greatest common divisor of $a$ and $b$.  For all sets $X$, let $\# X$ denote the cardinality of $X$. For all sets $X$ and $\Lambda$, let $X^\Lambda$ denote the set of all sequences $(x_\lambda)_{\lambda \in \Lambda}$ composed of $x_\lambda\in X$ for all $\lambda\in \Lambda$.  Define $\mathrm{e}(x)$ by $e^{2\pi \sqrt{-1}x} $ for all $x\in \mathbb{R}$.
\end{notation}

\section{Preparations}\label{Section-Preparations}

A sequence $(x_n)_{n=1}^\infty\in \mathbb{R}^\mathbb{N}$ is called \textit{uniformly distributed modulo 1} if for every $0\leq a<b\leq 1$, we have
\begin{equation}\label{Equation-udm1}
\lim_{N \rightarrow \infty} \frac{\#\left\{ n\in \mathbb{N}\cap[1,N] \colon \{x_n\} \in [a,b) \right\}}{N}=b-a.
\end{equation}
For further details on uniform distribution theory, see the book written by Kuipers and Niederreiter \cite{KuipersNiederreiter}. It is useful to calculate the decay of higher-order derivatives of $f(x)$ in order to verify that a sequence $(f(n))_{n\in \mathbb{N}}$ is uniformly distributed modulo $1$. For example, by \cite[Theorem~3.5]{KuipersNiederreiter}, let $k$ be a positive integer, and let $f(x)$ be a function defined for $x \geq  1$, which is $k$ times differentiable for $x \geq  x_0$. If $f^{(k)}(x)$ tends monotonically to $0$ as $x\rightarrow \infty$ and if $\lim_{x\rightarrow \infty} x| f^{(k)} (x) |= \infty$, then the sequence $(f (n))_{n\in \mathbb{N}}$ is uniformly distributed modulo $1$. By this theorem, we have 

\begin{example}\label{Example-udm1}For all $A\in \mathbb{R}\setminus \{0\}$ and non-integral $\alpha>1$, the sequence $(An^\alpha)_{n=1}^\infty$ is uniformly distributed modulo $1$.
\end{example}

\begin{proof}
Let $f(x)=Ax^{\alpha}$, and $k=\lfloor \alpha \rfloor+1$. Then $f^{(k)}(x)$ tends monotonically to $0$. In addition, there exists $C=C(\alpha)>0$ such that $|xf^{(k)}(x)|\geq C x^{\alpha-\lfloor \alpha \rfloor}\rightarrow \infty$ as $x\rightarrow\infty$. Therefore the sequence $(f(n))_{n=1}^\infty$ is uniformly distributed modulo $1$.
\end{proof}

To show Theorem~\ref{Theorem-main1}, it is not sufficient to only consider the uniform distribution of sequences. We also need to analyze the convergence speed of \eqref{Equation-udm1}. Let  $(x_n)_{ n=1}^ N$ be a sequence composed of $x_n\in \mathbb{R}$ for all $1\leq n\leq N$. We define the \textit{discrepancy} of $(x_n)_{n=1}^N$ by
\[
\mathcal{D}((x_n)_{n=1}^N )=\sup_{0\leq a<b\leq 1} \left| \frac{\#\left\{ n\in \mathbb{N}\cap[1,N] \colon \{x_n\} \in [a,b) \right\}}{N}-  (b-a)  \right|.
\]
We can find upper bounds of the discrepancy from evaluating exponential sums by the following inequality: for all $m\in \mathbb{N}$ 
\begin{equation}\label{Inequality-Erdos_Turan}
\mathcal{D}((x_n)_{n=1}^N)\leq    \frac{6}{m+1}
+\frac{2}{\pi}\sum_{h=1}^m \left(\frac{1}{h}-\frac{1}{m+1}\right) \left|\frac{1}{N} \sum_{n=1}^N \mathrm{e}( hx_n) \right|.  
\end{equation}
This is called the Erd\H{o}s-Tur\'an inequality, a proof of which can be found in \cite[Theorem~2.5 in Chapter~2]{KuipersNiederreiter}. If we can write $hx_n=f(n)$ in \eqref{Inequality-Erdos_Turan} for some smooth real function $f(x)$, then we can evaluate the right-hand side of \eqref{Inequality-Erdos_Turan} by the following lemma.
\begin{lemma}[van der Corput's $k$-th derivative test]\label{Lemma-van_der_Corput}
Let $f(x)$ be real and have continuous derivatives up to the $k$-th order, where $k\geq 4$. Let $\lambda_k \leq f^{(k)} (x) \leq h \lambda_k$ $($or the same as for $-f^{(k)}(x))$. Let $b-a\geq 1$. Then there exists $C(h,k)>0$ such that 
\[
\left|\sum_{a<n\leq b} \mathrm{e} (f(n)) \right| \leq C(h,k) \left((b-a)\lambda_k^{1/(2^k-2)}+(b-a)^{1-2^{2-k}}\lambda_k^{-1/(2^k-2)} \right).
\] 
\end{lemma}
\begin{proof}
See the book written by Titchmarsh \cite[Theorem~5.13]{Titchmarsh}.
\end{proof}

We next introduce the Hausdorff dimension. For every $U\subseteq \mathbb{R}$, state the diameter of $U$ as $\diam(U)=\sup_{x,y\in U}|x-y|$. Fix $\delta>0$. For all $F\subseteq \mathbb{R}$ and $s\in (0,1]$, we define
\[
\cH_{\delta}^s (F)=\inf \left\{\sum_{j=1}^\infty \diam(U_j)^s\colon F\subseteq \bigcup_{j=1}^\infty U_j,\  \diam(U_j)\leq \delta  \text{ for all }j\in \mathbb{N}  \right\},
\]
and $\cH^s(F)=\lim_{\delta\rightarrow +0} \cH_\delta^s(F)$ is called the $s$-\textit{dimensional Hausdorff measure} of $F$. Further, we define the \textit{Hausdorff dimension} of $F$ by
\[
\Haus F=\inf \{s\in (0,1] \colon \mathcal{H}^s(F)=0 \}.
\]
From the definition, the following basic properties hold: 
\begin{itemize}
\item (monotonicity) for all $F\subseteq E \subseteq \mathbb{R}$, $\Haus F \leq \Haus E$;
\item (bi-Lipschitz invariance) let $f:F\to \mathbb{R}$ be a bi-Lipschitz map, that is, there exist $C_1,C_2>0$ such that $C_1|x-y|\leq |f(x)-f(y)|\leq C_2|x-y|$ for all $x,y\in F$. Then $\Haus F=\Haus f(F)$.
\end{itemize}
Comprehensive details concerning fractal dimensions can be found in the book by Falconer \cite{Falconer}. By the second property and the mean value theorem, we immediately obtain
\begin{lemma}\label{Lemma-C1_invariance}
Let $U\subseteq \mathbb{R}$ be an open set and let $V\subseteq U$ be a compact set. Let $f:U\to \mathbb{R}$ be a continuously differentiable function satisfying $|f'(x)|>0$ for all $x\in V$. Then for all $F\subseteq V$, $\dim_{\mathrm{H}} f(F)=\dim_{\mathrm{H}} F$.
\end{lemma}

For all $\gamma\geq 2$ and sets $X\subseteq  \mathbb{R}$, define
\begin{align*}
\mathcal{A}(X,\gamma)=\{x\in X \colon &\text{there are infinitely many $(p,q) \in \mathbb{Z}\times \mathbb{N} $}\\
 &\text{such that } \left|x-\frac{p}{q} \right|\leq \frac{1}{q^\gamma}\ \}.
\end{align*}
In particular, if $X=\mathbb{R}$ and $\gamma=2$, we know that $\mathcal{A}(\mathbb{R},2)=\mathbb{R}$. This result is called Dirichlet's approximation theorem. In addition, in the case when $\gamma>2$ and $X=[0,1]$, the following result is known:

\begin{theorem}[Jarn\'{i}k's theorem]\label{Theorem-Jarnik}
For all $\gamma> 2$, we have $\dim_{\mathrm{H}} \mathcal{A}([0,1],\gamma)=2/\gamma.$
\end{theorem}
\begin{proof}
See \cite[Theorem~10.3]{Falconer}. 
\end{proof}
 In Section~\ref{Section-Propositions}, we will use rational approximations of $a^{1/\alpha}$, and we find infinitely many solutions $(x,y)\in \PS(\alpha)^2$ of the equation $y=ax+b$.

\section{Lemmas}
The goal of this section is to show a series of lemmas so as to evaluate discrepancies and calculate the Hausdorff dimension. 

We write $O(1)$ for a bounded quantity. If this bound depends only on some parameters $a_1,\ldots, a_n$, then for instance we write $O_{a_1,a_2,\ldots, a_n}(1)$. As is customary, we often abbreviate $O(1)X$ and $O_{a_1,\ldots, a_n}(1)X$ to $O(X)$ and $O_{a_1,\ldots, a_n}(X)$ respectively for a non-negative quantity $X$. We also state $f(X) \ll g(X)$ and $f(X) \ll_{a_1,\ldots, a_n} g(X)$  as $f(X)=O(g(X))$ and $f(X)=O_{a_1,\ldots, a_n}(g(X))$ respectively, where $g(X)$ is non-negative. 

\begin{lemma}\label{Lemma-discrepancy}
 For every non-integral $\alpha>1$, integer $k\geq 4$, and real numbers $\eta>0 $ and $V\geq 1$, if $\eta V^{\alpha-k}<1$ holds, then we have
\[
\mathcal{D}((\eta n^\alpha)_{V<n\leq 2V} )\ll_{\alpha,k} (\eta V^{\alpha-k})^{1/(2^k-1)} +\eta^{-1/(2^k-2)}V^{(k-\alpha)/(2^k-2) -2^{2-k}}.
\]
\end{lemma}

\begin{proof} Fix any $\alpha$, $k$, $\eta$, $V$ given in Lemma~\ref{Lemma-discrepancy} which satisfy $\eta V^{\alpha-k}<1$.  Let $f_h(x)=h\eta x^\alpha$ for every $h\in \mathbb{N}$ and $x>0$. Then we have
$
h\eta V^{\alpha-k}\ll_{\alpha,k} f^{(k)}(x) \ll_{\alpha,k}h\eta V^{\alpha-k}
$
for all $V<x\leq 2V$. Therefore, the following holds from the Erd\H{o}s-Tur\'an inequality \eqref{Inequality-Erdos_Turan} and Lemma~\ref{Lemma-van_der_Corput} with $f=f_h$: for all $m\in \mathbb{N}$,
\begin{align*}
\mathcal{D}((\eta n^\alpha )_{V<n\leq 2V})&\ll m^{-1} +\sum_{h=1}^{m} \frac{1}{h} \left| \frac{1}{V} \sum_{V<x\leq 2V} \mathrm{e}(h\eta x^\alpha ) \right|\\
&\ll_{\alpha, k} m^{-1} +\sum_{h=1}^m \frac{1}{h} \left|\left(h\eta V^{\alpha-k}\right)^{1/(2^k-2)}+V^{-2^{2-k}} \left(h \eta V^{\alpha-k}\right)^{-1/(2^k-2)}\right|\\
&\ll_{\alpha,k} m^{-1} +(m\eta V^{\alpha-k})^{1/(2^k-2)} +\eta^{-1/(2^k-2)}V^{(k-\alpha)/(2^k-2) -2^{2-k}}.
\end{align*}
Hence by substituting $m=\lceil (\eta^{-1} V^{k-\alpha})^{1/(2^k-1)} \rceil $, we get the lemma.
\end{proof}

\begin{lemma}\label{Lemma-discrepancy_of_B}
Let $\alpha>1$ be a non-integral real number, $\gamma\in \mathbb{R}$ with $0<\gamma-\alpha<1$, and let $A>0$ be a real number. Then there exist $Q_0=Q_0(\alpha,\gamma,A)>0$, $\xi_0=\xi_0(\alpha,\gamma)>0$, and $\psi=\psi(\alpha,\gamma)<0$ such that for all $Q\geq Q_0$ and $0<\xi\leq \xi_0$, we have  
\[
\mathcal{D}((A Q^\alpha x^\alpha )_{V<x\leq 2V})\ll_{\alpha,\gamma,A} Q^{\psi}
\]
where $V=Q^{(\gamma-\alpha-\xi)/\alpha}$.
\end{lemma}

\begin{proof}Take an integer $k\geq 4$ such that
\begin{equation}\label{Inequality-alpha_k}
\frac{\gamma (k-3)} { \gamma+k-3}< \alpha <\frac{\gamma k} { k+\gamma}. 
\end{equation}
Note that there exists such an integer $k$. Let $g(k )=\gamma (k-3)/(\gamma+k-3)$. Then $g(k)$ is strictly increasing for all $k\geq 4$. Since $g(4)=\gamma/(\gamma+1)$ and $\lim_{k\rightarrow \infty } g(k)=\gamma$, we have
\[
\alpha\in (1,\gamma) \subseteq \bigcup_{k=4}^\infty (g(k), g(k+3)).
\]
Therefore, there exists an integer $k\geq 4$ satisfying \eqref{Inequality-alpha_k}. Let us fix such an integer as $k=k(\alpha,\gamma)\geq 4$. By the condition $V= Q^{(\gamma-\alpha-\xi)/\alpha }$, we observe that
\[
A Q^\alpha V^{\alpha-k}=AQ^{\psi_1}
\]
where $\psi_1:=\alpha +(\gamma-\alpha-\xi)(\alpha-k)/\alpha$. Then the inequality $\alpha<\gamma k/(k+\gamma)$ yields that 
\[
\psi_1 <((\gamma+k)\alpha -\gamma k)/\alpha<0. 
\]
Therefore, if $Q_0$ is sufficiently large and $Q\geq Q_0$, then
$A Q^\alpha V^{\alpha-k}<1$. Thus we may apply Lemma~\ref{Lemma-discrepancy} with $\eta=A Q^\alpha$ and $V=Q^{(\gamma-\alpha-\xi)/\alpha }$ to obtain
\begin{align*}
\mathcal{D}((A Q^\alpha x^\alpha )_{V<x\leq 2V})&\ll_{\alpha,\gamma,A} (Q^{\alpha}V^{\alpha-k})^{1/(2k-1)}+Q^{-\alpha/(2^k-2)}V^{(k-\alpha)/(2^k-2) -2^{2-k}}  \\
&= Q^{\psi_1/(2^k-1)} +Q^{\psi_2}  \label{Inequality-discrepancy_on_q_n},
\end{align*}
where 
\[
\psi_2:=-\frac{\alpha}{2^k-2}+\frac{\gamma-\alpha-\xi}{\alpha} \left(\frac{k-\alpha}{2^k-2}-\frac{4}{2^k}  \right).
\]
Then we have
\begin{align*}
\psi_2&=\frac{-\alpha^22^k +(\gamma-\alpha)(k-\alpha)2^k -4(\gamma-\alpha)(2^k-2) }{2^k(2^k-2)\alpha}+O_{\alpha,\gamma}(\xi)\\
&=\frac{(\gamma k-(\gamma+k-4)\alpha-4\gamma )2^k +8(\gamma-\alpha) }{2^k(2^k-2)\alpha}+O_{\alpha,\gamma}(\xi).
\end{align*}
 Therefore the inequalities $\alpha> \gamma(k-3)/(\gamma+k-3)$ and $\gamma-\alpha<1$ imply that for sufficiently small $\xi>0$,
\begin{align*}
\psi_2&<\frac{-2^{k}\cdot \gamma^2 /(\gamma+k-3) +8  }{2^k(2^k-2)\alpha}+O_{\alpha, \gamma}(\xi)\\
&<\frac{-2^{k}/(k-2) +8  }{2^{k+1}(2^k-2)\alpha}\leq 0.
\end{align*}
Therefore, there exists $\psi=\psi(\alpha,k)<0$ so that  $\mathcal{D}((A Q^\alpha x^\alpha )_{V<x\leq 2V})\ll_{\alpha, \gamma, A} Q^{\psi}$.
\end{proof}
We next present lemmas on the Hausdorff dimension.
\begin{lemma}\label{Lemma-GenJarnik}
For all non-empty and bounded open intervals $J\subseteq \mathbb{R}$, we have
 \[
 \dim_{\mathrm{H}} \mathcal{A}(J,\gamma)=2/\gamma.
 \]
\end{lemma}

\begin{proof}There exist $m\in \mathbb{Z}$ and $h\in \mathbb{N}$ such that $J\subseteq [m,m+h]$. When $x\in \mathcal{A} (J,\gamma)$, there are infinitely many $(p,q) \in \mathbb{Z}\times \mathbb{N}$ such that $\left|x-p/q \right|\leq q^{-\gamma}$. Then for all $\epsilon>0$, and for infinitely many $(p,q)\in \mathbb{Z}\times \mathbb{N}$,
\[
\left| \frac{x-m}{h} - \frac{p-mq }{qh} \right|\leq \frac{1}{hq^{\gamma}}\leq \frac{1}{q^{\gamma-\epsilon}}.
\]
Thus $f(\mathcal{A}(J,\gamma)) \subseteq \mathcal{A}([0,1],\gamma-\epsilon)$ where $f(x)=(x-h)/m$. By the bi-Lipschitz invariance and monotonicity of the Hausdorff dimension and Theorem~\ref{Theorem-Jarnik}, we obtain 
\[
\Haus \mathcal{A}(J,\gamma) = \Haus  f(\mathcal{A}(J,\gamma)) \leq \dim_{\mathrm{H}} \mathcal{A}([0,1],\gamma-\epsilon)=\frac{2}{\gamma-\epsilon}.
\]
By taking $\epsilon\rightarrow +0$, $\Haus \mathcal{A}(J,\gamma) \leq 2/\gamma$.

We next show that $\Haus \mathcal{A}(J,\gamma) \geq 2/\gamma$. There exist $\ell \in \mathbb{Z}$ and $M\in \mathbb{N}$ such that $J\supseteq [\ell/M, (\ell+1)/M]$. Take such $\ell$ and $M$.  Then for all $x\in \mathcal{A} ([0,1],\gamma+\epsilon)$, there are infinitely many $(p,q) \in \mathbb{Z}\times \mathbb{N}$ such that $\left|x-p/q \right|\leq q^{-\gamma-\epsilon}$. Then for all $\epsilon>0$, and for infinitely many $(p,q)\in \mathbb{Z}\times \mathbb{N}$, we have
\[
\left|\frac{\ell+x}{M} -\frac{\ell q+p}{qM} \right|\leq \frac{1}{M}\frac{1}{q^{\gamma+\epsilon}} <\frac{1}{q^\gamma}.
\]
This inequality and $(\ell+x)/M\in J$ imply that $g(\mathcal{A}([0,1],\gamma+\epsilon)) \subseteq \mathcal{A} (J,\gamma)$ where $g(x)=(x+\ell)/M$.  By the monotonicity and bi-Lipschitz invariance of the Hausdorff dimension and Theorem~\ref{Theorem-Jarnik}, we have $\Haus \mathcal{A} (J,\gamma ) \geq 2/(\gamma+\epsilon)$ for all $\epsilon>0$. Hence by taking $\epsilon\rightarrow +0$, we conclude $\Haus \mathcal{A} (J,\gamma ) \geq 2/\gamma$.
\end{proof}

\begin{lemma}\label{Lemma-application_Jarnik}
Let $I\subseteq (1,\infty)$ be a non-empty and bounded open interval, and let $\gamma> 2$ and $a>0$ be real numbers with $a\neq 1$. Define
\begin{align*}
\mathcal{E}(I,\gamma;a)=\{\alpha\in I \colon &\text{there are infinitely many $(p,q)\in \mathbb{Z}\times \mathbb{N}$} \\
&\text{ such that } \left|a^{1/\alpha}-\frac{p}{q} \right|\leq \frac{1}{q^{\gamma}}\}.
\end{align*}
Then we have 
$
\Haus \mathcal{E}(I,\gamma;a)=2/\gamma.
$
\end{lemma}

\begin{proof}
For all $u>0$, let $f(u)=a^{1/u}$. Fix a compact set $V\subseteq \mathbb{R}$ with $I \subseteq V$. Clearly, $f$ is continuously differentiable and $|f'(u)|>0$ for all $u\in V$.  By the definitions, $f(\mathcal{E}(I,\gamma;a))=\mathcal{A}(f(I),\gamma)$. Since $f(I)$ is also a bounded open interval, Lemma~\ref{Lemma-C1_invariance} and Lemma~\ref{Lemma-GenJarnik} imply that
\[
\dim_{\mathrm{H}} \mathcal{E}(I,\gamma; a)=\dim_{\mathrm{H}} f(\mathcal{E}(I,\gamma;a))=\dim_{\mathrm{H}} \mathcal{A}(f(I),\gamma)=\frac{2}{\gamma}.  
\]
\end{proof}

\section{Key Propositions}\label{Section-Propositions}
In this section, we show two key propositions by applying rational approximations. 
\begin{proposition}\label{Proposition-upper}Let $a,b\in \mathbb{R}$ with $a\neq 1$ and $a>0$. For all $1\leq \beta <\gamma$, we have 
\[
\{\alpha\in (\beta,\gamma) \colon y=ax+b \text{ is solvable in } \mathrm{PS}(\alpha) \}
\subseteq \mathcal{E}((\beta,\gamma),\beta; a).
\]
\end{proposition}

\begin{proof}Fix $\beta,\gamma\in \mathbb{R}$ with $1\leq \beta <\gamma$. Take any $\alpha\in (\beta,\gamma)$ such that the equation $y=ax+b$ is solvable in $\mathrm{PS}(\alpha)$. Then there are infinitely many $(p,q)\in \mathbb{N}\times \mathbb{N}$ such that $\lfloor p^\alpha \rfloor  = a\lfloor q^\alpha \rfloor+ b$, which implies that
\[
\frac{p}{q}=\left(a+\frac{b+\{p^\alpha\}-a\{q^\alpha\}}{ q^\alpha}\right)^{1/\alpha} =a^{1/\alpha} + O_{a,b}(q^{-\alpha}).
\]
Hence, there exist $C=C(a,b)>0$ such that for infinitely many $(p,q)\in \mathbb{N}^2$, 
\[
\left|a^{1/\alpha} - \frac{p}{q}\right| \leq \frac{C}{q^{\alpha}} \leq \frac{1}{q^{\beta}}.
\]
This yields that $\alpha \in \mathcal{E}((\beta,\gamma),\beta;a)$.
\end{proof}

\begin{proposition}\label{Proposition-lower}Let $a,b\in \mathbb{R}$ with $a\neq 1$ and $0\leq b<a$. Suppose that $y=ax+b$ is solvable in $\mathbb{N}$. Then for all $1\leq \beta <\gamma$ with $\lfloor \beta \rfloor < \beta<\gamma < \lfloor \beta \rfloor+1$, we have 
\[
 \mathcal{E}((\beta,\gamma),\gamma;a)\subseteq \{\alpha\in (\beta,\gamma) \colon y=ax+b \text{ is solvable in } \mathrm{PS}(\alpha) \}.
\]
\end{proposition}

\begin{proof}
Since the equation $y=ax+b$ is solvable in $\mathbb{N}$, there exist distinct solutions $(x_1,y_1), (x_2,y_2)\in \mathbb{N}^2$ of the equation. Then since $y_2-y_1=a(x_2-x_1)$ and $(x_1,y_1)\neq (x_2,y_2)$, we have $a\in \mathbb{Q}$. In addition, $b\in \mathbb{Q}$ holds from $b=y_1-ax_1$. Thus we may let $a=a_1/a_2$, $b=b_1/b_2$, ($a_1,a_2,b_2\in\mathbb{N}$, $b_1\in \mathbb{N}\cup \{0\}$).  By letting $c=a_2b_2$, $d=a_1b_2$, $e=a_2b_1$, a pair $(x,y)\in \mathbb{N}^2$ satisfies the equation $cy-dx=e$ if and only if $(x,y)$ satisfies the equation $y=ax+b$. Therefore we now discuss the solvability in $\PS(\alpha)$ of the equation $cy-dx=e$. Take any $\alpha\in \mathcal{E}((\beta,\gamma),\gamma;a)=\mathcal{E}((\beta,\gamma),\gamma;d/c)$. Let us show that the equation $cy-dx=e$ is solvable in $\PS(\alpha)$.

By the definition, there is a sequence $((p_n,q_n))_{n=1}^\infty \in (\mathbb{Z}\times \mathbb{N})^\mathbb{N}$ such that for all $n\in \mathbb{N}$, 
\[
\left|(d/c)^{1/\alpha}-p_n/q_n\right|<{q_n^{-\gamma}},
\]
where $q_1<q_2<\cdots$. Since $(d/c)^{1/\alpha}>0$ and $d/c\neq 1$, there exists $n_0=n_0(d,c)\in \mathbb{N}$ such that for all $n\geq n_0$, we obtain $p_n>0$ and $p_n\neq q_n$.

From the solvability in $\mathbb{N}$, there exist $u,v\in \mathbb{N}$ such that $cu-dv=e$. By the division algorithm, there exist $r,v'\in \mathbb{Z}$ such that $v=cr+v'$, $0\leq v'<c$. Hence by replacing $u-dr$ and $v'$ with $u$ and $v$ respectively, we obtain
\begin{equation}\label{Equation-cde}
cu-dv=e,\quad 0\leq v<c.
\end{equation}

Take sufficiently small parameters $\xi=\xi(\alpha,\gamma)>0$ and $\epsilon\in (0,1-e/d)$, and take a sufficiently large parameter $n_1=n_1(\alpha,\gamma,c,d,\epsilon)\in \mathbb{N}$. Note that we verify the existence of $\epsilon$ since $e<d$. Take $n \in \mathbb{N}$ with $n\geq n_1$. Let $V_n= q_n^{(\gamma-\alpha-\xi)/\alpha } $. Define 
\begin{gather*}
I=\left[ \frac{v}{c},\ \frac{v}{c}+\frac{1}{c} \right) \cap \left[\frac{u}{d}+\frac{\epsilon}{d},\ \frac{u}{d}+\frac{1}{d}-\frac{\epsilon}{d} \right),\ B_n=\left\{x\in \mathbb{N} \colon \left\{ \frac{(q_nx)^\alpha}{c}  \right\}\in I  \right\}.
\end{gather*}
If $n_1$ is large enough and $\xi$ is small enough, then by the definition of the discrepancy and Lemma~\ref{Lemma-discrepancy_of_B} with $V=V_n$, $Q=q_n$, $A=1/c$, there exists $\psi=\psi(\alpha,\gamma)<0$ such that 
\[
\#(B_n\cap (V_n,2V_n])/V_n = \mathrm{diam}(I\cap [0,1) ) +O_{\alpha,\gamma,c}(q_{n}^{\psi}).
\]
Here we show $\mathrm{diam}(I\cap [0,1) )>0$. Indeed, by \eqref{Equation-cde}, we obtain 
\[
\frac{u}{d} +\frac{\epsilon}{d}=\frac{v}{c}+\frac{e}{cd}+\frac{\epsilon}{d}> \frac{v}{c}\geq 0.
\]
Moreover, the inequality $\epsilon<1-e/d$ yields that
\[
\frac{u}{d}+\frac{\epsilon}{d} = \frac{v}{c}+\frac{1}{c} + \frac{e-d}{cd}+\frac{\epsilon}{c} < \frac{v}{c}+\frac{1}{c}\leq 1.
\]
Hence $\mathrm{diam}(I\cap [0,1) )>0$. Therefore there exists a large enough $n_1=n_1(\alpha,\gamma,c, d,e,\epsilon)\in \mathbb{N}$ such that for all $n\geq n_1$, we get $\#(B_n\cap (V_n,2V_n])/V_n \geq  \mathrm{diam}(I\cap [0,1) )/2$,
which means that $B_n\cap (V_n,2V_n]$ is non-empty.

Hence we may take $x\in B_n\cap (V_n,2V_n]$ where $n\geq n_1$. Then
\[
 (q_nx)^\alpha= c  \left\lfloor \frac{(q_nx)^\alpha}{c} \right\rfloor +c \left\{ \frac{(q_nx)^\alpha}{c} \right\}
 = \left(c  \left\lfloor \frac{(q_nx)^\alpha}{c} \right\rfloor+v\right) +\left(c \left\{ \frac{(q_nx)^\alpha}{c} \right\}-v\right).
\]
The first term on the most right-hand side is an integer, and the second is in $[0,1)$ from the definition of $B_n$. Therefore we have
\[
\lfloor (q_nx)^\alpha \rfloor=   c  \left\lfloor \frac{(q_nx)^\alpha}{c} \right\rfloor+v.
\]
Let $\theta=p_n/q_n-(d/c)^{1/\alpha}$. By the mean value theorem, there exist $C=C(c,d,\alpha)>0$ and $\theta'\in \mathbb{R}$ with $|\theta'|\leq |\theta|$ such that
$
\left(p_n/q_n\right)^\alpha=( \left(d/c\right)^{1/\alpha} +\theta )^\alpha=d/c+C\theta'.
$
Therefore,
\begin{align*}
(p_n x)^\alpha&=\left(\frac{p_n}{q_n}\right)^\alpha (q_n x)^\alpha =d \left\lfloor \frac{(q_nx)^\alpha}{c} \right\rfloor+ d \left\{ \frac{(q_nx)^\alpha}{c} \right\}+C\theta' (q_nx)^\alpha\\
&=\left(d \left\lfloor \frac{(q_nx)^\alpha}{c} \right\rfloor+u\right) + \left(d \left\{ \frac{(q_nx)^\alpha}{c} \right\}-u\right)+C\theta' (q_nx)^\alpha.
\end{align*}
The first term on the most right-hand side is an integer, and the second term is in $[\epsilon,1-\epsilon)$ by $x\in B_n$. Further, if necessary, we replace $n_1$ with a larger one, and by $x\in (V_n,2V_n]$, the third term is evaluated by
\[
|C\theta' (q_nx)^\alpha| \leq 2^\alpha C\frac{q_n^\alpha q_n^{\gamma-\alpha-\xi} }{ q_n^\gamma } \leq 2^\alpha Cq_n^{-\xi}\leq 2^\alpha Cq_{n_1}^{-\xi}< \epsilon.
\]
 Hence we obtain
\[
\lfloor (p_nx)^\alpha \rfloor =d \left\lfloor \frac{(q_nx)^\alpha}{c} \right\rfloor+u.
\]
By the above discussion, if $x\in B_n\cap(V_n,2V_n]$ and $n\geq n_1$, then 
\[
c\lfloor (p_nx)^\alpha \rfloor-d\lfloor(q_nx)^\alpha \rfloor= cd  \left\lfloor \frac{(q_nx)^\alpha}{c} \right\rfloor+cu
-dc  \left\lfloor \frac{(q_nx)^\alpha}{c} \right\rfloor-dv=cu-dv=e,
\]
which means that $(\lfloor (q_nx)^\alpha \rfloor, \lfloor (p_nx)^\alpha \rfloor)\in \mathbb{N}^2$ is a solution of the equation $cy-dx=e$. Therefore the equation $cy-dx=e$ is solvable in $\mathrm{PS}(\alpha)$ since $B_n\cap (V_n ,2V_n]$ is non-empty for all $n\geq n_1$.
\end{proof}

\section{Proof of Theorem~\ref{Theorem-main1}}\label{Section-Proof_of_Theorem}
Fix $a,b\in \mathbb{R}$ with $a\neq 1$ and $0\leq b<a$. In the case $\alpha \in (1,2)$, we apply Proposition~\ref{Proposition-lower} with $\beta=1$ and $\gamma=2$. Then
\[
\mathcal{E}((1,2),2;a)\subseteq \{\alpha \in (1,2) \colon \text{$y=ax+b$ is solvable in $\mathrm{PS}(\alpha)$} \}.
\]
By Dirichlet's approximation theorem, $\mathcal{E}((1,2),2;a)=(1,2)$. Therefore the equation $y=ax+b$ is solvable in $\mathrm{PS}(\alpha)$ for all $\alpha\in (1,2) $.


We next discuss the case when $\alpha>2$. Fix $s,t\in \mathbb{R}$ with $2<s<t$. By applying Proposition~\ref{Proposition-upper} with $\beta=s$ and $\gamma=t$, and applying Lemma~\ref{Lemma-application_Jarnik}, we have
\begin{equation}\label{Inequality-upper_bound_for_main1}
\dim_{\mathrm{H}} \{\alpha \in (s,t) \colon \text{$y=ax+b$ is solvable in $\mathrm{PS}(\alpha)$} \}\leq\dim_{\mathrm{H}} \mathcal{E}((s,t),t;a)= \frac{2}{s}.
\end{equation}
Further, let $\delta>0$ be an arbitrarily small parameter. By applying Proposition~\ref{Proposition-lower} with $\beta=s$ and $\gamma=\min\{s+\delta,\lfloor s\rfloor +1, t \}$, and applying Lemma~\ref{Lemma-application_Jarnik}, we obtain
\begin{align*}
\dim_{\mathrm{H}} \{\alpha \in (s,t) \colon \text{$y=ax+b$ is solvable in $\mathrm{PS}(\alpha)$} \} \geq \dim_{\mathrm{H}} \mathcal{E}((s,\gamma),\gamma;a)=  \frac{2}{s+\delta}
\end{align*}
for every small enough $\delta>0$. Therefore we get the theorem by taking $\delta\rightarrow +0$.

\begin{remark} Let $\alpha \in (\beta,\gamma)$ where $\beta$ and $\gamma$ satisfy $1\leq \beta <\gamma$ and $\lfloor \beta \rfloor <\beta <\gamma <\lfloor\beta \rfloor+1$. If $a^{1/\alpha} \in \mathbb{Q}$, then it is clear that for infinitely many $(p,q)\in \mathbb{Z}\times \mathbb{N}$ we have $| a^{1/\alpha}-p/q | \leq q^{-\gamma}$. By Proposition~\ref{Proposition-lower}, the equation $y=ax+b$ is solvable in $\PS(\alpha)$. Therefore, for all $a,b\in \mathbb{R}$ with $a\neq 1$ and $0\leq b<a$, and for all non-integral $\alpha>1$ satisfying $a^{1/\alpha} \in \mathbb{Q}$, the equation $y=ax+b$ is solvable in $\PS(\alpha)$. 
\end{remark}

\begin{remark}We apply Proposition~\ref{Proposition-upper} and Lemma~\ref{Lemma-application_Jarnik} to show the inequality \eqref{Inequality-upper_bound_for_main1}. Note that the condition $0\leq b<a$ is not required in Proposition~\ref{Proposition-upper} and Lemma~\ref{Lemma-application_Jarnik}. Hence, for all $a, b\in \mathbb{R}$ with $a\neq 1$ and $a>0$, and for all $s,t\in \mathbb{R}$ with $2< s<t$, we obtain
\[
\dim_{\mathrm{H}} \{\alpha \in (s,t) \colon \text{$y=ax+b$ is solvable in $\mathrm{PS}(\alpha)$} \}\leq \frac{2}{s}.
\]
\end{remark}

\section*{Acknowledgement}
The author was supported by JSPS KAKENHI Grant Number JP19J20878.
\bibliographystyle{amsalpha}
\bibliography{references_PS}
\end{document}